\documentclass[review,onefignum,onetabnum]{siamart250211}

\usepackage{lipsum}
\usepackage{amsfonts}
\usepackage{graphicx}
\usepackage{epstopdf}
\usepackage{algorithmic}
\usepackage{amssymb}
\usepackage[english]{babel}
\usepackage{array}
\usepackage{adjustbox}
\usepackage{relsize}
\usepackage{spverbatim}
\usepackage{listings}
\usepackage{mathrsfs}
\usepackage{cleveref}
\usepackage{physics}

\DeclareMathAlphabet{\mathpzc}{OT1}{pzc}{m}{it}

\newcolumntype{L}{>{$}l<{$}}

\ifpdf
  \DeclareGraphicsExtensions{.eps,.pdf,.png,.jpg}
\else
  \DeclareGraphicsExtensions{.eps}
\fi

\newsiamremark{hypothesis}{Hypothesis}
\crefname{hypothesis}{Hypothesis}{Hypotheses}
\newsiamthm{claim}{Claim}
\newtheorem{remark}{Remark}
\nolinenumbers

\headers{Reverse Generalised Bessel Polynomials}{T. M. Dunster}

\title{Simplified Airy function Asymptotic expansions for Reverse Generalised Bessel Polynomials}

\author{T. M. Dunster\thanks{Department of Mathematics and Statistics, San Diego State University, 5500 Campanile Drive, San Diego, CA 92182-7720, USA. 
  (\email{mdunster@sdsu.edu}, \url{https://tmdunster.sdsu.edu}).}
  }

\usepackage{amsopn}

\makeatletter
\newcommand*{\addFileDependency}[1]{
  \typeout{(#1)}
  \@addtofilelist{#1}
  \IfFileExists{#1}{}{\typeout{No file #1.}}
}
\makeatother

\nolinenumbers

\ifpdf
\hypersetup{
  pdftitle={Simplified Airy function Asymptotic expansions for Reverse Generalised Bessel Polynomials},
  pdfauthor={T. M. Dunster}
}
\fi

\begin{document}

\maketitle

\begin{abstract}
Uniform asymptotic expansions are derived for reverse generalised Bessel polynomials of large degree $n$, real parameter $a$, and complex argument $z$, which are simpler than previously known results. The defining differential equation is analysed; for large $n$ and $\frac{3}{2} - n < a < \infty$, it possesses two turning points in the complex $z$ plane which are complex conjugates. Away from these turning points Liouville–Green expansions are obtained for the polynomials and two companion solutions of the differential equation, where asymptotic series appear in the exponent. Then representations involving Airy functions and two slowly varying coefficient functions are constructed. Using the Liouville–Green representations, asymptotic expansions are obtained for the coefficient functions that involve coefficients that can be easily and explicitly computed recursively. In conjunction with a suitable re-expansion, or Cauchy's integral formula, near the turning point, the expansions are valid for $-\Delta_{1} n+\frac{3}{2} \leq a \leq \Delta_{2} n$ for fixed arbitrary $\Delta_{1} \in (0,1)$ and bounded positive $\Delta_{2}$, uniformly for all unbounded complex values of $z$.
\end{abstract}



\section{Introduction}
\label{sec:Introduction}

The generalised Bessel polynomials are defined by
\begin{equation} 
\label{eq01}
y_{n}(z;a)=\sum_{k=0}^{n}\binom{n}{k}(n+a-1)_{k}\left(\tfrac12 z\right)^{k},
\end{equation}
where $(\alpha)_{k}=\Gamma(\alpha+k)/\Gamma(\alpha)$ is Pochhammer’s symbol. These were first introduced by Bochner and Romanovsky \cite{Bochner:1929:SLP,Romanovsky:1929:NCP}. Their importance became more widely recognised following the studies of Krall and Frink \cite{Krall:1949:NCO}, who established their connection to solutions of the spherical wave equation, and of Thompson \cite{Thompson:1949:DNF}, who employed them in the analyses of electrical delay networks. 

In this paper we focus on the reverse generalised Bessel polynomials which are given by
\begin{equation} 
\label{eq02}
\theta_{n}(z;a)=z^{n}y_{n}(z^{-1};a).
\end{equation}
We consider the asymptotic behaviour of $\theta_{n}(z;a)$ for large degree $n$ and complex argument $z$. Results in this direction had remained underdeveloped for some time, chiefly due to the analytical difficulties introduced by turning points in the associated differential equations, which, for $a>\frac32-n$ and $a \neq 2$, are neither purely real nor imaginary. Complex turning points complicated both integral and differential equation methods of asymptotic analysis. For the exceptional case when $a = 2$ the turning points become purely imaginary, the polynomials can be expressed in terms of modified Bessel functions \cite[Eqs. (1.4) and (1.12)]{Dunster:2001:GBP}, and classical asymptotic results can be used (cf. \cite[Chap.~11, Sect.~10]{Olver:1997:ASF}).

Among the more recent investigations, Wong and Zhang \cite{Wong:1997:AEG} provided Airy-type asymptotic expansions for $y_n(z; a)$ based on an integral representation and the method of Chester, Friedman, and Ursell \cite{Chester:1957:EMS}. More recently, the present author \cite{Dunster:2001:GBP} provided asymptotic approximations for $\theta_{n}(z;a)$ valid over wider domains in the complex $z$ plane—specifically for regions including both the origin and infinity—as well as for a broader range of the parameter $a$. Additionally, the expansions obtained included explicit error bounds. This was achieved by an application of two general asymptotic theories, due to F. W. J. Olver \cite[Chaps. 10 and 11]{Olver:1997:ASF}, to an ordinary differential equation satisfied by $\theta_{n}(z;a)$: one for the case of complex domains which are free of turning points, yielding Liouville–Green (LG) expansions, and the other for complex domains containing a simple turning point (yielding Airy function expansions). The approximations are uniformly valid for $z$ lying in certain unbounded subdomains of the complex plane and include explicit error bounds. Together, the domains of validity cover the whole complex $z$ plane.

The drawback in using Olver's turning point expansion is that the coefficients satisfy recursive formulas that require nested integrations, which in most cases, and certainly in the present one, become prohibitively complicated. This, in turn, leads to error bounds which, while important theoretically, are not practicable to evaluate. In an application to Bessel functions this was addressed in \cite{Dunster:2017:COA} in which LG expansions were used with the asymptotic series in the exponent of the approximating exponential function \cite{Dunster:1998:AOT,Dunster:2020:LGE}. We use this approach here.

Returning to the reverse generalised Bessel polynomials $\theta_{n}(z;a)$, they satisfy the linear second-order differential equation
\begin{equation} 
\label{eq03}
z\frac{d^{2}\theta}{dz^{2}}
-(2z+2n+a-2)\frac{d\theta}{dz}+2n\theta=0,
\end{equation}
which has a regular singularity at $z = 0$ and an irregular singularity at infinity. The roles of these two singularities are reversed for the differential equation satisfied by $y_{n}(z;a)$; asymptotic results are more straightforward to establish if finite singularities are regular, and this is the reason we focus on $\theta_{n}(z;a)$.

Following \cite[Eq. (2.3)]{Dunster:2001:GBP} define
\begin{equation} 
\label{eq04}
w_{n}^{(0)}(z;a)=2^{-n-a+1}z^{1-n-\frac{1}{2}a} 
e^{-z} \theta_{n}(z;a).
\end{equation}
Then from \cref{eq02,eq03,eq04} $w=w_{n}^{(0)}(z;a)$ satisfies
\begin{equation} 
\label{eq05}
\frac{d^{2}w}{dz^{2}}=\left\{1+\frac{a-2}{z}
+\frac{(2n+a)(2n+a-2)}{4z^{2}}\right\}w,
\end{equation}
and is recessive at infinity in the right half-plane, since
\begin{equation} 
\label{eq06}
w_{n}^{(0)}(z;a)=2^{-n-a+1}z^{1-\frac{1}{2}a} e^{-z}
\left\{1+\mathcal{O}(z^{-1})\right\}
\quad (z \to \infty),
\end{equation}
whereas all other solutions of \cref{eq04} are dominant in the half-plane $|\arg(z)|<\frac12 \pi$. In \cite[Eqs. (2.6) and (2.10)]{Dunster:2001:GBP} two numerically satisfactory companion solutions are defined by
\begin{equation} 
\label{eq07}
w_{n}^{(1)}(z;a)=(-1)^{n+1}z^{n+\frac12 a}e^{-z}
V(n+a-1,2n+a,2z),
\end{equation}
and
\begin{equation} 
\label{eq08}
w_{n}^{(-1)}(z;a)=z^{n+\frac12 a}e^{-z}
\mathbf{M}(n+a-1,2n+a,2z),
\end{equation}
where $V(a,b,z)$ and $\mathbf{M}(a,b,z)$ are confluent hypergeometric functions (cf. \cite[pp. 255-256]{Olver:1997:ASF}). On referring to \cite[Eq. (2.9)]{Dunster:2001:GBP} we have for the former
\begin{equation} 
\label{eq09}
w_{n}^{(1)}(z;a)=2^{-n-1}z^{\frac{1}{2}a-1} e^{z}
\left\{1+\mathcal{O}(z^{-1})\right\}
\quad \left(z \to \infty,\, |\arg(-z)|
\leq \tfrac32 \pi - \delta\right),
\end{equation}
where $\delta$ is an arbitrary small positive constant, and hence is recessive at infinity in the left half-plane $|\arg(-z)|< \frac12 \pi$.

From \cite[Eq. (2.12)]{Dunster:2001:GBP}  
\begin{equation} 
\label{eq10}
w_{n}^{(-1)}(z;a)=\frac{z^{n+\frac{1}{2}a}}{\Gamma(2n+a)}
\left\{1+\mathcal{O}(z)\right\}
\quad (z \to 0),
\end{equation}
which demonstrates that it is the unique solution of \cref{eq04} that is recessive at $z=0$. From \cref{eq08} and \cite[Eq. 13.7.2]{NIST:DLMF} we also note that
\begin{multline} 
\label{eq11}
w_{n}^{(-1)}(z;a)
= \frac{z^{\frac12 a-1}e^{z}}{2^{n+1}\Gamma(n+a-1)}
\left\{1+\mathcal{O}\left(\frac{1}{z}\right)\right\}
\\
\quad \left(z \to \infty, \, 
|\arg(z)|\leq \tfrac12 \pi - \delta\right).
\end{multline}

An important connection formula that we shall use is given by \cite[Eq. (2.13)]{Dunster:2001:GBP} 
\begin{equation} 
\label{eq12}
w^{(-1)}_{n}(z; a) = (-1)^{n+1} \frac{e^{a\pi i}}{n!} 
w^{(0)}_{n}(z; a) 
+ \frac{1}{\Gamma(n + a - 1)} w^{(1)}_{n}(z; a).
\end{equation}
This verifies that the three are linearly independent for $a>1-n$, which is the case in this paper.

The plan of this paper is as follows. In \cref{sec:LG}, we derive LG expansions for all three solutions of the differential equation under consideration. These expansions are valid within specific domains of the upper half-plane, although they are not applicable near the turning point. The coefficients of these expansions can be readily and recursively determined; they are given explicitly as polynomials in certain trigonometric functions. We make use of the results from \cite{Dunster:2020:LGE}, which also provide simple and effective error bounds. Extension of our results into the lower half-plane comes immediately from the Schwarz reflection principle, since all three of our solutions are analytic in the upper half-plane and real on the positive real axis.

In \cref{sec:Airy}, we apply the method developed by \cite{Dunster:2017:COA,Dunster:2021:SEB} to obtain simplified expansions for differential equations with turning points. These expansions are expressed in terms of Airy functions and two slowly varying coefficient functions. Asymptotic expansions for these coefficient functions are also obtained and are shown to depend on the coefficients of the LG expansions in \cref{sec:LG}. Complete expansions are given for all three solutions, with the principal result for the function $\theta_n(z; a)$ presented in \cref{thm:ThetaAiry}. We numerically check the accuracy of the Airy expansion for $\theta_n(z; a)$ for real and complex values of $z$.

We mention that in a subsequent paper we plan to use the new methods of \cite{Dunster:2024:AZB} to study both a uniform asymptotic and numerical evaluation of the zeros of $\theta_{n}(z;a)$. This would extend a recent numerical study of these polynomials and their zeros carried out in \cite{Dunster:2021:CPB}.

\section{LG expansions}
\label{sec:LG}

From \cite[\S 3]{Dunster:2001:GBP} we express \cref{eq05} in the form
\begin{equation} 
\label{eq13}
\frac{d^2 w}{dz^2} 
= \left\{ u^2 f(\alpha, z) + g(z) \right\} w,
\end{equation}
where
\begin{equation} 
\label{eq14}
u= n+\frac12, \, \alpha = \frac{a - 2}{u},
\end{equation}
and
\begin{equation} 
\label{eq15}
f(\alpha, z) =
\frac{\left(z + \tfrac{1}{2} \alpha\right)^2 +1+\alpha}{z^{2}},
\quad
g(z) = -\frac{1}{4z^2}.
\end{equation}

For large $u$ the zeros of $f(\alpha,z)$ are the turning points of \cref{eq13}, and these are given by
\begin{equation} 
\label{eq16}
z_{1,2}(\alpha) 
= \pm i \sigma - \tfrac{1}{2} \alpha,
\end{equation}
where 
\begin{equation} 
\label{eq28}
\sigma=\sqrt{1+\alpha}.
\end{equation}
For $-\infty < \alpha < \infty$ the fundamental properties of these turning points can be summarised as follows.

\begin{itemize}
    \item $-\infty < \alpha < -1$, $\alpha \neq -2$, they are real and positive. We define the square roots so that $z_{1} <  z_{2}$.
    \item $\alpha  = -2 \implies z_{1} =0$ (coalesces with the pole) and $z_{2} =2$.
    \item $\alpha = -1$ they merge into a double turning point at $z=\frac12$.
    \item $-1 < \alpha < \infty$ they are complex conjugates.
\end{itemize}

Taking the above into account, we assume
\begin{equation} 
\label{eq17}
-1<-1+\delta \leq \alpha \leq \alpha_{1} < \infty,
\end{equation}
and, as such, the turning points are complex conjugates, bounded, and bounded away from the pole as well from each other. From \cref{eq14} we see that \cref{eq17} is equivalent to $-\Delta_{1} n+\frac{3}{2} \leq a \leq \Delta_{2} n$ for fixed arbitrary $\Delta_{1} \in (0,1)$ and bounded positive $\Delta_{2}$.

We note from \cref{eq16} that $z_{1,2}(\alpha) \to \pm i$ as $\alpha \to 0$ so that
\begin{equation} 
\label{eq18}
z_{1,2}(\alpha) 
= \pm i - \tfrac{1}{2}(1 \mp i)\alpha \mp \tfrac{1}{8} i \alpha^2 
+ \mathcal{O}(\alpha^3),
\end{equation}
and also
\begin{equation} 
\label{eq19}
|z_{1,2}(\alpha)| = 1 + \tfrac{1}{2} \alpha.
\end{equation}

For LG expansions we employ a variable $\xi$, and for Airy expansions a variable $\zeta$ (cf. \cite[Chaps. 10 and 11]{Olver:1997:ASF}), and these are defined by
\begin{equation} 
\label{eq20}
\frac{2}{3}\zeta^{3/2}=\xi = \int_{z_{1}(\alpha)}^{z} f^{1/2}(\alpha, t)\, dt.
\end{equation}
The lower limit was chosen so that $\zeta=\xi=0$ at the turning point $z=z_{1}$, and this is a requirement for our Airy function expansions of \cref{sec:Airy}. Upon integration, we obtain the explicit expression
\begin{multline}
\label{eq21}
\xi=Z-\left(1+\tfrac12 \alpha\right)
\ln\left\{\frac{4Z+2\alpha(Z+z+2)+4+\alpha^{2}}{z} \right\} 
\\
+\tfrac12 \alpha\ln(2Z+2z+\alpha)
+\tfrac{1}{2}\ln(1+\alpha)
+\left(2+\tfrac12 \alpha\right)\ln(2)
-\tfrac{1}{2}(1+\alpha)\pi i,
\end{multline}
where
\begin{equation} 
\label{eq22}
Z = Z(\alpha,z) =\left\{(z - z_1)(z - z_2)\right\}^{1/2}
= \left\{\left(z + \tfrac{1}{2} \alpha\right)^2 
+1+\alpha \right\}^{1/2}.
\end{equation}
The branch in \cref{eq22} is chosen so that $Z$ is positive for $0<z<\infty$, negative for $-\infty<z<0$, and continuous elsewhere in the upper half-plane $\Im(z) \geq 0$ having a cut along a path from $z=0$ to $z=z_{1}$ on which $\Im(\xi)=0$. Thus, in this cut region $Z \sim z$ as $z \to \infty$ for $\Im(z) \geq 0$. In addition, principal branches of the logarithms in \cref{eq21} are taken.

As a result $\Re(\xi) \to \pm \infty$ as $\Re(z) \to \pm \infty$, and specifically $\xi \to \infty$ as $z \to \infty$ so that
\begin{equation} 
\label{eq23}
\xi = z + \tfrac{1}{2} \alpha \ln(2z) 
+  \tfrac{1}{2} \alpha - \tfrac{1}{2}
(1 + \alpha) \ln(1 + \alpha) - \tfrac{1}{2}(1 + \alpha)\pi i 
+ \mathcal{O}(z^{-1}).
\end{equation}
Further, $\Re(\xi) \to -\infty$ as $z \to 0^{+}$ so that
\begin{equation} 
\label{eq24}
\xi=\left(1+\tfrac12 \alpha\right)\ln\left\{\frac{2ez}{(2+\alpha)^{2}}\right\}
+\tfrac12(1+\alpha)\left\{\ln(1+\alpha)-\pi i\right\}
+\mathcal{O}(z),
\end{equation}
and $\Re(\xi) \to +\infty$ as $z \to 0^{-}$ so that
\begin{equation} 
\label{eq25}
\xi=\left(1+\tfrac12 \alpha\right)\ln\left\{
\frac{(2+\alpha)^{2}}{2e|z|}\right\}
-\tfrac12(1+\alpha)\ln(1+\alpha)
-\tfrac12 \pi i
+\mathcal{O}(z).
\end{equation}

In \cref{fig:Figzplane,fig:Figxiplane1,fig:Figxiplane2} the $z-\xi$ map of the upper half-plane $0 \leq \arg(z) \leq \pi$ is depicted, with corresponding points labeled $\mathsf{A},\mathsf{B},\ldots,\mathsf{J}$; \cref{fig:Figxiplane1} shows the $\xi$ map of the region in the $z$ plane bounded by $\mathsf{A}\mathsf{B}\mathsf{C}\mathsf{D}\mathsf{E}\mathsf{F}\mathsf{A}$, and \cref{fig:Figxiplane2} shows the $\xi$ map of the region in the $z$ plane bounded by $\mathsf{A}\mathsf{F}\mathsf{G}\mathsf{H}\mathsf{I}\mathsf{J}\mathsf{A}$. 

Note that in \cref{fig:Figzplane} the curves $\mathsf{A}\mathsf{D}$, $\mathsf{A}\mathsf{H}$ and $\mathsf{A}\mathsf{F}$ emanating from the turning point $z_{1}$ are the so-called Stokes lines, defined by $\Re(\xi)=0$, and it is near these that the complex zeros of certain solutions lie. In particular, the zeros in the upper half-plane of $\theta_{n}(z;a)$ lie close to the curve $\mathsf{A}\mathsf{H}$.

\begin{figure}
 \centering
 \includegraphics[
 width=1.0\textwidth,keepaspectratio]
 {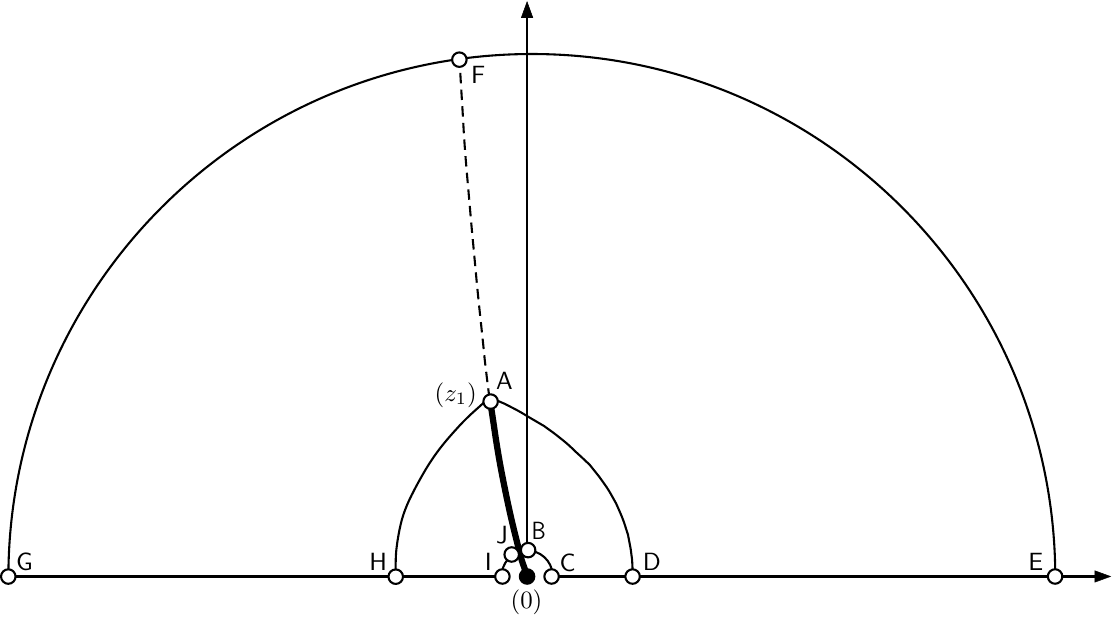}
 \caption{$z$ plane.}
 \label{fig:Figzplane}
\end{figure}

\begin{figure}
 \centering
 \includegraphics[
 width=1.0\textwidth,keepaspectratio]
 {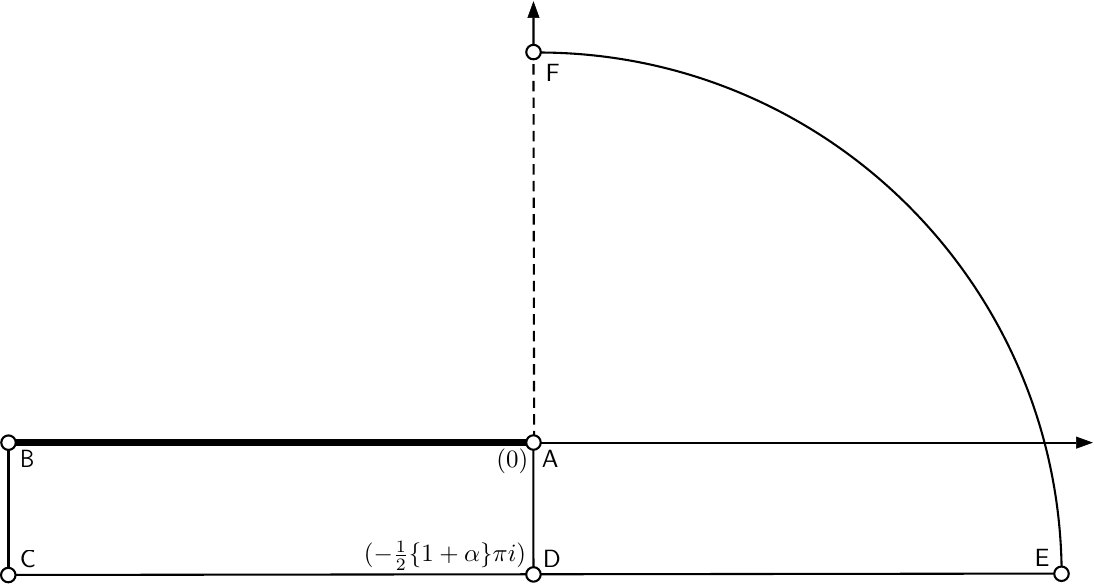}
 \caption{$\xi$ plane.}
 \label{fig:Figxiplane1}
\end{figure}

\begin{figure}
 \centering
 \includegraphics[
 width=1.0\textwidth,keepaspectratio]
 {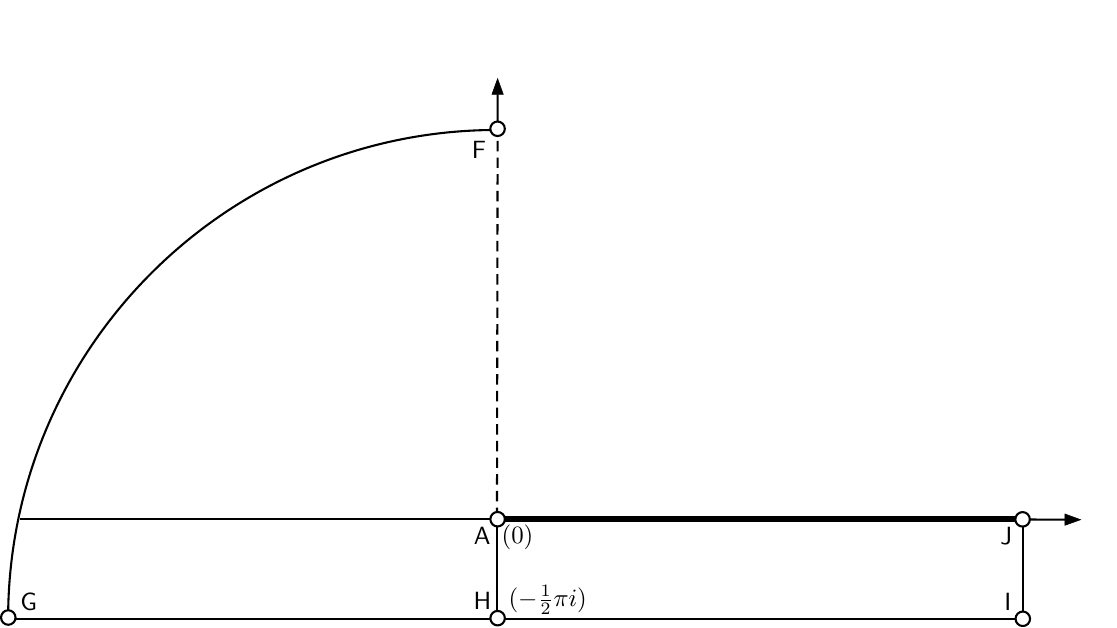}
 \caption{$\xi$ plane.}
 \label{fig:Figxiplane2}
\end{figure}

The LG expansions involve the Schwarzian derivative \cite[Eq. (1.6)]{Dunster:2020:LGE}, which we consider as a function of $z$. From \cref{eq15} it is given by
\begin{equation} 
\label{eq26}
\psi(z) = -\frac{z \left\{  4z^3-(4 + 3\alpha)(4 + \alpha)z 
- \alpha(2 + \alpha)^2 \right\}}
{16\left\{ \left(z + \tfrac{1}{2} \alpha\right)^2 
+1+\alpha\right\}^{3}}.
\end{equation}
Note that $\psi(z)=\mathcal{O}(z)$ as $z \to 0$, and  $\psi(z)=\mathcal{O}(z^{-2})$ as $z \to \infty$, and as such our LG expansions will be valid at both these singularities of \cref{eq13}. On the other hand, it is unbounded at both turning points $z=z_{1,2}$, and therefore the LG expansions will not be valid at either. As noted, expansions valid at $z=z_{1}$ in terms of Airy functions will be constructed in \cref{sec:Airy}.

We now consider the evaluation of the LG coefficients given by \cite[Eqs. (1.10) - (1.12)]{Dunster:2020:LGE}. To aid in the integrations required we follow \cite[Eqs. (6.1) and (6.3)]{Dunster:2001:GBP} (but with a different notation)) and introduce a parameter $\phi$ by
\begin{equation} 
\label{eq27}
\sin(\phi)=\frac{\sigma}{Z},
\end{equation}
where $\sigma$ and $Z$ are given by \cref{eq28,eq22}. Equivalently, on referring to \cref{eq22}, we have
\begin{equation} 
\label{eq29}
\cos(\phi)
=\frac{z+\tfrac{1}{2} \alpha}{Z}.
\end{equation}

With the branch of $Z$ as described above, we note that both $\sin(\phi)$ and $\cos(\phi)$ are positive for $z>0$ and continuous elsewhere in the upper half-plane $\Re(z) \geq 0$ having the cut $\mathsf{A}\mathsf{B}$, $\mathsf{A}\mathsf{J}$ from $z=0$ to $z=z_{1}$ depicted in \cref{fig:Figzplane,fig:Figxiplane1,fig:Figxiplane2}. In particular, $\cos(\phi)>0$ for $-\infty<z<-\frac12 \alpha$, and $\cos(\phi) \to 1$ as $z \to \infty$ along any ray in the upper half-plane. We note that, from dividing \cref{eq29} by \cref{eq27},
\begin{equation} 
\label{eq30}
z=\sigma\,\cot(\phi)
-\tfrac{1}{2}\alpha.
\end{equation}

On recalling that $d\xi/dz=f^{1/2}(\alpha,z)$ (cf. \cref{eq20}), we have from \cref{eq27,eq29,eq30}
\begin{equation} 
\label{eq31}
\frac{d\phi}{d\xi}
=-\frac{z\sin^{3}(\phi)}{1+\alpha}.
\end{equation}
Hence, from \cref{eq22,eq26,eq27,eq28,eq29,eq30,eq31} and \cite[Eqs. (1.10) - (1.12)]{Dunster:2020:LGE},
\begin{equation} 
\label{eq32}
\mathrm{E}_{1}(\alpha,\phi)
=\frac{\sin(\phi)\left\{5\cos^{2}(\phi)-2\right\}}
{24\,\sigma}+\frac{\alpha\left\{\cos(\phi)
\left(5\cos^{2}(\phi)-6\right)+1 \right\}}{48(1+\alpha)},
\end{equation}
\begin{multline} 
\label{eq33}
\mathrm{E}_{2}(\alpha,\phi)
=\frac{\alpha\cos(\phi)\sin^{3}(\phi)\left\{3-5\cos^{2}(\phi)\right\}}
{16\,\sigma^{3}}
\\
+\frac{\sin^{2}(\phi)}{64(1+\alpha)^{2}}
\left\{ 5\left(4-\alpha^{2}+4\alpha\right)\cos^{4}(\phi)
+(7\alpha^2 - 16\alpha - 16)\cos^{2}(\phi)
-2\alpha^{2} \right\},
\end{multline}
and for $s=2,3,4,\ldots$
\begin{equation} 
\label{eq34}
\mathrm{E}_{s+1}(\alpha,\phi) =
G(\alpha,\phi)
\frac{\partial \mathrm{E}_{s}(\alpha,\phi)}
{\partial \phi}
+\int_{0}^{\phi} 
G(\alpha,\varphi)\sum\limits_{j=1}^{s-1}
\frac{\partial \mathrm{E}_{j}(\alpha,\varphi)}{\partial \varphi}
\frac{\partial \mathrm{E}_{s-j}(\alpha,\varphi)}{\partial \varphi} d\varphi,
\end{equation}
where from \cref{eq30,eq31}
\begin{equation} 
\label{eq35}
G(\alpha,\phi)
=-\frac12 \frac{d\phi}{d\xi}
=\frac{\cos(\phi)\sin^{2}(\phi)}
{2 \sigma}
-\frac{\alpha \sin^{3}(\phi)}{4(1+\alpha)}.
\end{equation}

Our desired LG solutions are then given by \cite[Eqs. (1.16) - (1.18)]{Dunster:2020:LGE}, with the obvious change of notation. We actually have two of the form with $+u\xi$ in the exponent of the approximating exponential function, and two with $-u\xi$ in the exponent, but these all differ due to different error terms which depend on the so-called reference points, which are singularities of \cref{eq13} at which the chosen solution is recessive. Specifically, for $N=2,3,4,\ldots$, we construct LG solutions of the equation as given by
\begin{equation}
\label{eq36}
W_{0}(u,a,z) =
\exp \left\{ -u\xi +\sum\limits_{s=1}^{N-1}{(-1)^{s}
\frac{\mathrm{E}_{s}(\alpha,\phi)}{u ^{s}}}\right\} 
\left\{ 1+\eta_{N,0}(u,\alpha,z) \right\},
\end{equation}
\begin{equation}
\label{eq37}
W_{1}(u,a,z) =
\exp \left\{u\xi +\sum\limits_{s=1}^{N-1}{
\frac{\mathrm{E}_{s}(\alpha,\phi)}{u ^{s}}}\right\} 
\left\{ 1+\eta_{N,1}(u,\alpha,z) \right\},
\end{equation}
\begin{equation}
\label{eq38}
W_{-1}^{+}(u,a,z) =
\exp \left\{u \xi +\sum\limits_{s=1}^{N-1}
{\frac{\mathrm{E}_{s}(\alpha,\phi)}{u^{s}}}\right\} 
\frac{1+\eta_{N,-1}^{+}(u,\alpha,z)}
{1+\eta_{N,-1}^{+}(u,\alpha,\infty)},
\end{equation}
and 
\begin{equation}
\label{eq39}
W_{-1}^{-}(u,a,z) =
\exp \left\{-u \xi +\sum\limits_{s=1}^{N-1}(-1)^{s}
{\frac{\mathrm{E}_{s}(\alpha,\phi)}{u^{s}}}\right\} 
\frac{1+\eta_{N,-1}^{-}(u,\alpha,z)}
{1+\eta_{N,-1}^{-}(u,\alpha,\infty)}.
\end{equation}
The factors $1+\eta_{N,-1}^{\pm}(u,\alpha,\infty)$ in \cref{eq38,eq39} were introduced to make these two solutions independent of $N$ (as are the other two), as we shall show below.

Bounds for the $\eta$-error terms are given as follows. Using \cref{eq35} let
\begin{equation}
\label{eq40}
\mathrm{F}_{s}(\alpha,z)
=\frac{\partial \mathrm{E}_{s}(\alpha,\phi)}{\partial \xi}
=-2G(\alpha,\phi)
\frac{\partial \mathrm{E}_{s}(\alpha,\phi)}{\partial \phi},
\end{equation}
where the $\sin(\phi)$ and $\cos(\phi)$ terms appearing on the RHS are regarded as functions of $z$, as given by \cref{eq22,eq27,eq38}. Now let us choose the reference points that appear in the error bounds, and which determine the asymptotic regions of validity. Firstly, let $z=\alpha_{0}=+\infty$ ($\arg(z)=0$), and  $z=\alpha_{1}=-\infty$ ($\arg(z)=\pi$), which correspond to $\xi =\infty-\frac12(1+\alpha) \pi i$ and $\xi = -\infty-\frac12 \pi i$, respectively (see \cref{fig:Figxiplane1,fig:Figxiplane2}). Next, let $z=\alpha_{-1}^{\pm}=0^{\pm}$, i.e. the limit as $z$ approaches zero through positive and negative values, respectively; cf. the points  labeled $\mathsf{C}$ and $\mathsf{I}$ in \cref{fig:Figzplane} to the right and left of the cut.

Now, from \cref{eq15,eq20,eq22},
\begin{equation}
\label{eq41}
\frac{d\xi}{dz}=\frac{Z(\alpha,z)}{z}.
\end{equation}
Then we have from \cite[Thm. 1.1]{Dunster:2020:LGE} for $j=0,1$ and $N=2,3,4,\ldots$
\begin{equation}
\label{eq42}
\left\vert \eta_{N,j}(u,\alpha,z)\right\vert  
\leq \frac{1}{u^{N}} \Phi_{N,j}(u,\alpha,z) 
\exp \left\{\frac{1}{u }\Psi_{N,j}(u,\alpha,z) 
+\frac{1}{u^{N}} \Phi_{N,j}(u,\alpha,z)\right\},
\end{equation}
where
\begin{multline}
\label{eq43}
\Phi_{N,j}(u,\alpha,z) 
=2\int_{\alpha _{j}}^{z}
{\left\vert {t^{-1}Z(\alpha,t)  
\mathrm{F}_{N}(\alpha,t) dt}\right\vert } 
\\
+\sum\limits_{s=1}^{N-1}{\frac{1}{u^{s}}
\sum\limits_{k=s}^{N-1}
\int_{\alpha _{j}}^{z}
\left\vert{t^{-1}Z(\alpha,t)
\mathrm{F}_{k}(\alpha,t) 
\mathrm{F}_{s+N-k-1}(\alpha,t) dt}\right\vert},
\end{multline}
and
\begin{equation}
\label{eq44}
\Psi_{N,j}(u,\alpha,z) 
=4\sum\limits_{s=0}^{N-2}\frac{1}{u^{s}}
\int_{\alpha _{j}}^{z}
{\left\vert {t^{-1}Z(\alpha,t)
\mathrm{F}_{s+1}(\alpha,t) dt}\right\vert }.
\end{equation}
The error terms $\eta_{N,-1}^{\pm}(u,\alpha,z)$ also satisfy these bounds, with $\alpha_{j}$ replaced by $\alpha_{-1}^{\pm}$.

The paths for all the integrals (i) consist of a finite chain of $R_{2}$ arcs (as defined in \cite[Chap. 5, \S 3.4]{Olver:1997:ASF}), and (ii) $\Re\{\xi(t)\}$ is monotonic as $t$ passes along the path from $\alpha_{j}$ (or $\alpha_{-1}^{\pm}$) to $z$, where $\xi(t)$ is given by \cref{eq21,eq22} with $z=t$. For our purposes, paths which correspond to polygons in the $\xi$ plane suffice for all these bounds. 

For $j=0,1$ let the regions where the bounds for $\eta_{N,j}(u,\alpha,z)$ hold in the $z$ plane be denoted by $Z_{j}$, with $\Xi_{j}$ denoting the corresponding ones in the $\xi$ plane. Similarly for $\eta_{N,-1}^{\pm}(u,\alpha,z)$, let these regions be denoted by $Z_{-1}^{\pm}$, with $\Xi_{-1}^{\pm}$ denoting the corresponding ones in the $\xi$ plane. Then $Z_{j}$ consists of all points $z$ that can be linked to $\alpha_{j}$ by such a path and for which the integrals converge (and likewise for $Z_{-1}^{\pm}$). They do converge at the singularities $z=0$ and $z=\infty$, but not at the turning point $z=z_{1}$ ($\xi=0$), and hence the paths must be bounded away from this point. 

Now it is seen that $Z_{0}$ includes all points in the upper half-plane $0 \leq \arg(z) \leq \pi$ except the closed region bounded by $\mathsf{A}\mathsf{H}\mathsf{I}\mathsf{J}\mathsf{A}$. Similarly, $Z_{1}$ includes all points in the same upper half-plane except the closed region bounded by $\mathsf{A}\mathsf{B}\mathsf{C}\mathsf{D}\mathsf{A}$. Furthermore, $Z_{-1}^{+}$ is the region bounded by $\mathsf{A}\mathsf{B}\mathsf{C}\mathsf{D}\mathsf{E}\mathsf{F}\mathsf{A}$, and $Z_{-1}^{-}$ is the region bounded by $\mathsf{A}\mathsf{F}\mathsf{G}\mathsf{I}\mathsf{J}\mathsf{A}$, where in both cases points on the bounding Stokes curve $\mathsf{A}\mathsf{F}$ being excluded. 

For all four regions it is understood that the points $\mathsf{B}$, $\mathsf{C}$, $\mathsf{I}$ and $\mathsf{J}$ are arbitrarily close to $z=0$, and the points $\mathsf{E}$, $\mathsf{F}$ and $\mathsf{G}$ are arbitrarily large. 

Before we match the LG expansions with the three fundamental solutions, we define three domains $Z_{j,k}=Z_{j}\cap Z_{k}$ which we will be important in the next section on Airy function expansions. These are the unshaded regions illustrated in \cref{fig:Z0UZ1,fig:Zm1UZO,fig:Zm1UZ1}, with the dashed boundaries open (they are not included in the regions).

\begin{figure}
 \centering
 \includegraphics[
 width=1.0\textwidth,keepaspectratio]
 {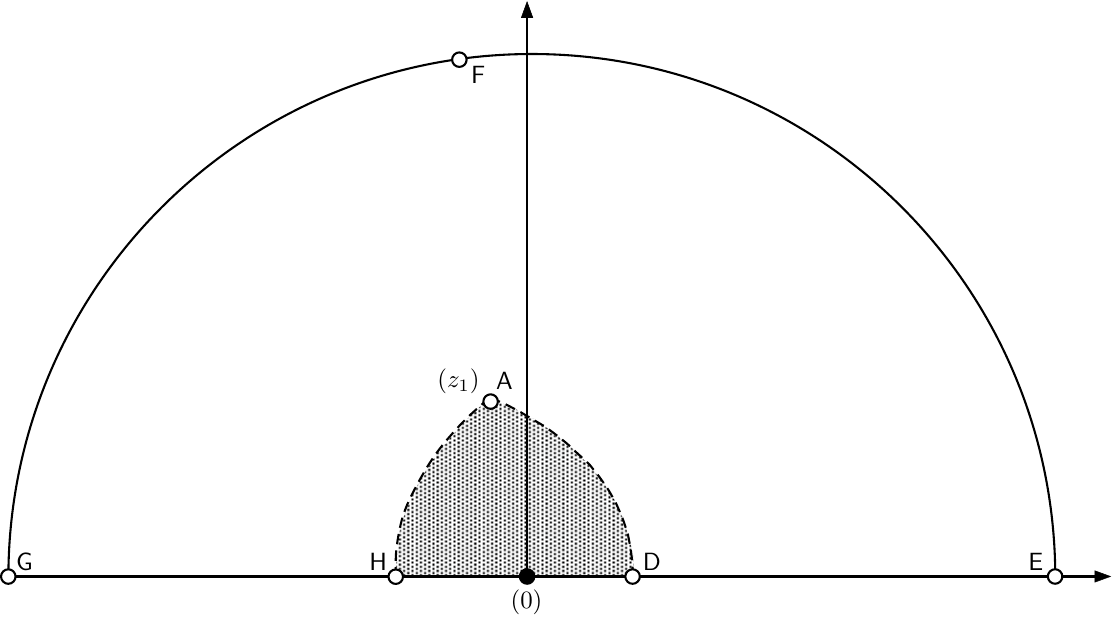}
 \caption{$z$ plane: region $Z_{0,1}$}
 \label{fig:Z0UZ1}
\end{figure}

\begin{figure}
 \centering
 \includegraphics[
 width=1.0\textwidth,keepaspectratio]
 {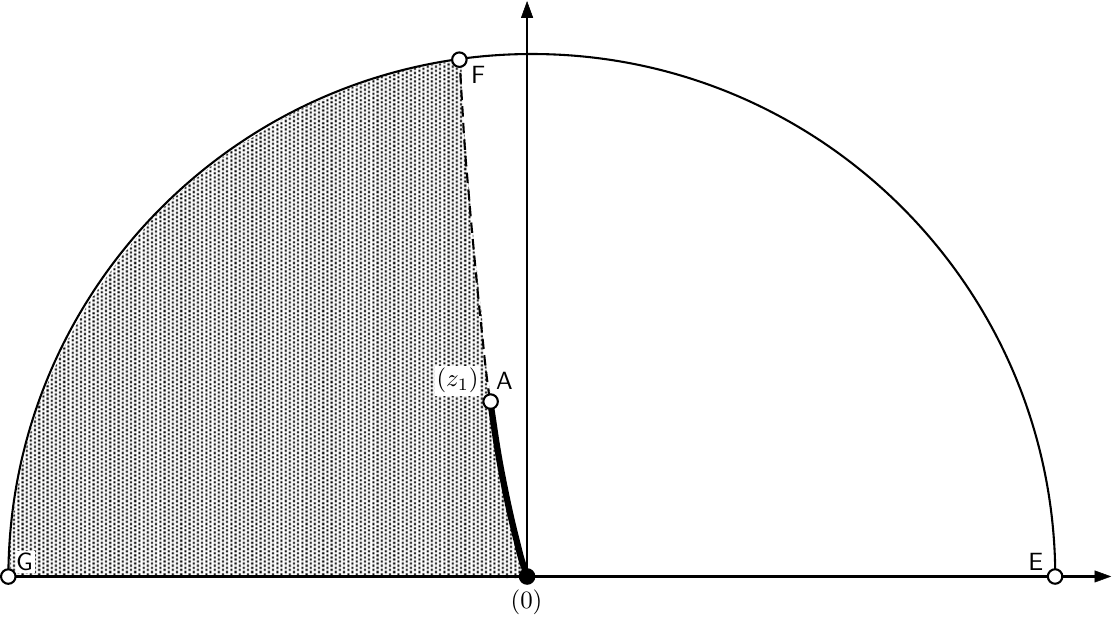}
 \caption{$z$ plane: region $Z_{-1,0}$.}
 \label{fig:Zm1UZO}
\end{figure}

\begin{figure}
 \centering
 \includegraphics[
 width=1.0\textwidth,keepaspectratio]
 {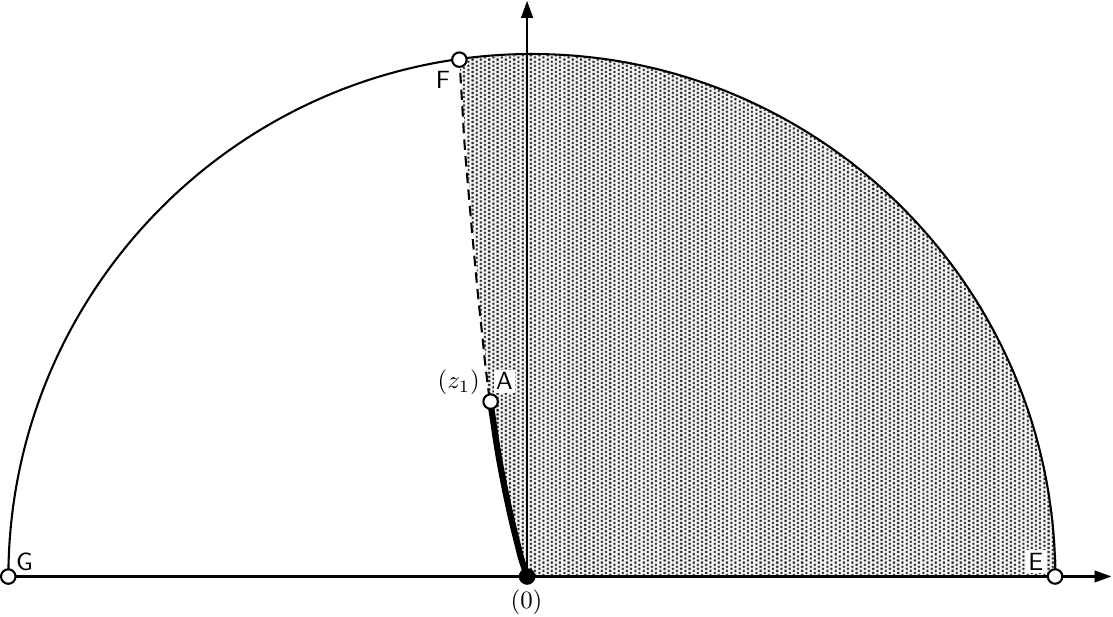}
 \caption{$z$ plane: region $Z_{-1,1}$.}
 \label{fig:Zm1UZ1}
\end{figure}

Let us now match the LG expansions with $w_{n}^{(j)}(uz;a)$ ($j=0,\pm1$). Firstly, $\lim_{z \to \infty} \eta_{N,0}(u,\alpha,z)=0$. Furthermore, from \cref{eq27}, we note that $\phi \to 0$ as $z \to \infty$, and hence from \cref{eq32,eq33,eq34,eq35} $\lim_{z \to \infty}\mathrm{E}_{s}(\alpha,\phi)=\mathrm{E}_{s}(\alpha,0)=0$. Thus, from \cref{eq36}, $W_{0}(u,a,z)$ has the same unique (recessive) behaviour at $\Re(z)=+\infty$ regardless of the value of $N$, and therefore is independent of this parameter. Likewise for $W_{1}(u,a,z)$ and $W_{-1}^{\pm}(u,a,z)$ with their recessive properties at $\Re(z)=-\infty$ and $z=0^{\pm}$, respectively; for the latter, note from \cref{eq24,eq25} that $\exp(\pm u\xi) \to 0$ as $z \to 0^{\pm}$.

Now, for solutions recessive at $z=\alpha_{0}=+\infty$, from \cref{eq06,eq15,eq23,eq36} we have
\begin{equation} 
\label{eq45}
w_{n}^{(0)}(uz;a)
=C_{n}^{(0)}(a) f^{-1/4}(a,z)W_{0}(u,a,z),
\end{equation}
where, from comparing both sides as $z \to \infty$,
\begin{equation} 
\label{eq46}
C_{n}^{(0)}(a)
=\frac{e^{-u(1+\alpha)\pi i/2} }{\sqrt{2}\,\{4(1+\alpha)\}^{u/2}}
\left\{\frac{e}{2u(1+\alpha)}\right\}^{u\alpha/2}.
\end{equation}

We do a similar matching of the solutions of \cref{eq13} with $f^{-1/4}(a,z)W_{1}(u,a,z)$ and $f^{-1/4}(a,z)W_{-1}(u,a,z)$, again by finding the multiplicative constants by comparison at $z=\infty$ (even though the latter is not recessive at infinity). For the former, we have for solutions recessive at $z=\alpha_{1}=-\infty$
\begin{equation} 
\label{eq47}
w_{n}^{(1)}(uz;a)
=C_{n}^{(1)}(a) f^{-1/4}(a,z)W_{1}(u,a,z),
\end{equation}
for some constant $C_{n}^{(1)}(a)$. As we remarked, to determine this we again compare both sides at $z=\infty$. From \cref{eq09,eq15,eq23,eq37,eq47} and the fact that $\lim_{z \to \infty} \eta_{N,1}(u,a,z)=0$, we find that
\begin{equation} 
\label{eq48}
C_{n}^{(1)}(a)
=\frac{e^{u(1+\alpha)\pi i/2} }{\sqrt{2}}
\left(\frac{1+\alpha}{4}\right)^{u/2}
\left\{\frac{u(1+\alpha)}{2e}\right\}^{u\alpha/2}.
\end{equation}

Finally, the LG solutions $f^{-1/4}(a,z)W_{-1}^{\pm}(u,a,z)$ as given by \cref{eq38,eq39} are both recessive at $z=0$, and hence to within a multiplicative constant are the same, and both proportional to $w_{n}^{(-1)}(uz;a)$. Thus we have
\begin{equation} 
\label{eq49}
w_{n}^{(-1)}(uz;a)
=C_{n}^{(-1)\pm}(a)f^{-1/4}(a,z)W_{-1}^{\pm}(u,a,z),
\end{equation}
where from \cref{eq11,eq15,eq23,eq38}
\begin{equation} 
\label{eq50}
C_{n}^{(-1)+}(a)
=\frac{e^{u(1+\alpha)\pi i/2} (1+\alpha)^{u/2}}
{2^{n+1}\Gamma(n+a-1)}
\left\{\frac{u(1+\alpha)}{2e}\right\}^{u\alpha/2}.
\end{equation}
Similarly, from \cref{eq06,eq09,eq12,eq15,eq23,eq39},
\begin{equation} 
\label{eq51}
C_{n}^{(-1)-}(a)
=\frac{ie^{u(1+\alpha)\pi i/2}}{\sqrt{2}\,
\left\{4(1+\alpha)\right\}^{u/2}n!}
\left\{\frac{e}{2u(1+\alpha)}\right\}^{u\alpha/2}.
\end{equation}

\section{Airy expansions}
\label{sec:Airy}

Following \cite{Dunster:2017:COA} for some appropriately selected functions $\mathcal{A}(u,a,z)$ and $\mathcal{B}(u,a,z)$ let us define
\begin{equation}
\label{eq52}
Y_{l}(u,a,z)
=\mathrm{Ai}_{l} \left(u^{2/3}\zeta\right)\mathcal{A}(u,a,z)
+\mathrm{Ai}'_{l} \left(u^{2/3}\zeta \right)\mathcal{B}(u,a,z)
\quad (l=0,\pm 1),
\end{equation}
where $\zeta$ is given by \cref{eq20,eq21}, and we have used Olver's notation for complex-valued Airy functions $\mathrm{Ai}_{j}(z) =\mathrm{Ai}(z e^{-2\pi ij/3})$ ($j=0,\pm 1$) (see \cite[\S 9.2(iii)]{NIST:DLMF}). From (\cite[Eq. 9.2.12]{NIST:DLMF}) we note that
\begin{equation}
\label{eq53}
\mathrm{Ai}(z) 
=e^{\pi i/3}\mathrm{Ai}_{1}(z) 
+e^{-\pi i/3}\mathrm{Ai}_{-1}(z).
\end{equation}

Now $Y_{0}(u,a,z)$ is recessive as $z \to \alpha_{0} = +\infty$ since $\Re(\zeta) \to +\infty$ in this limit (cf. \cite[Eq. 9.7.5]{NIST:DLMF}). In addition, $Y_{1}(u,a,z)$ is recessive as $z \to \alpha_{1} = -\infty$ since $\Re(\zeta e^{-2\pi i/3}) \to +\infty$ in this limit. We therefore define $\mathcal{A}(u,a,z)$ and $\mathcal{B}(u,a,z)$ uniquely using the pair of equations
\begin{equation}
\label{eq54}
(-1)^{n} \frac{e^{a\pi i}}{n!} w^{(0)}_{n}(uz; a) 
= e^{-\pi i/3}
\left\{\frac{\zeta}{f(a,z)}\right\}^{1/4} Y_{0}(u,a,z),
\end{equation}
and
\begin{equation}
\label{eq55}
\frac{1}{\Gamma(n + a - 1)} w^{(1)}_{n}(uz; a)
=\left\{\frac{\zeta}{f(a,z)}\right\}^{1/4} Y_{1}(u,a,z),
\end{equation}
noting the similar recessive behaviour at the respective singularities in both cases. The constants appearing in these equations were chosen so that
\begin{equation} 
\label{eq56}
w^{(-1)}_{n}(z; a) = e^{\pi i/3}
\left\{\frac{\zeta}{f(a,z)}\right\}^{1/4} Y_{-1}(u,a,z),
\end{equation}
which comes from \cref{eq12,eq53}. The significance is that both functions in this relation share the same recessive property at $z=\alpha_{-1}=0$ ($\Re(\zeta e^{2\pi i/3}) = +\infty$).

Now, from \cref{eq52} and the Airy functions Wronskian \cite[Eq. 9.2.8]{NIST:DLMF}, on choosing any pair from \cref{eq54,eq55,eq56} and solving the system, we arrive at the following three representations for $\mathcal{A}(u,a,z)$ and $\mathcal{B}(u,a,z)$:
\begin{multline} 
\label{eq57}
\mathcal{A}(u,a,z)=2\pi\left\{\frac{f(a,z)}{\zeta}\right\}^{1/4}
\left\{ (-1)^{n} \frac{e^{(a+\frac16)\pi i}}{n!}
w^{(0)}_{n}(uz; a)\mathrm{Ai}_{\pm 1}'\left(u^{2/3}\zeta\right) 
\right.
\\
\left. -\lambda_{\pm 1} 
w^{(\pm 1)}_{n}(uz; a)\mathrm{Ai}'\left(u^{2/3}\zeta\right) 
\right\},
\end{multline}
\begin{multline} 
\label{eq58}
\mathcal{A}(u,a,z)=2\pi\left\{\frac{f(a,z)}{\zeta}\right\}^{1/4}
\left\{ \frac{i}{\Gamma(n + a - 1)}
w^{(1)}_{n}(uz; a)\mathrm{Ai}_{-1}'\left(u^{2/3}\zeta\right) 
\right.
\\
\left. -e^{\pi i/6}w^{(-1)}_{n}(uz; a)\mathrm{Ai}_{1}'\left(u^{2/3}\zeta\right) 
\right\},
\end{multline}
\begin{multline} 
\label{eq59}
\mathcal{B}(u,a,z)=-2\pi\left\{\frac{f(a,z)}{\zeta}\right\}^{1/4}
\left\{ (-1)^{n} \frac{e^{(a+\frac16)\pi i}}{n!}
w^{(0)}_{n}(uz; a)\mathrm{Ai}_{\pm 1}\left(u^{2/3}\zeta\right) 
\right.
\\
\left. -\lambda_{\pm 1}
w^{(\pm 1)}_{n}(uz; a)\mathrm{Ai}\left(u^{2/3}\zeta\right) 
\right\},
\end{multline}
and
\begin{multline} 
\label{eq60}
\mathcal{B}(u,a,z)=-2\pi\left\{\frac{f(a,z)}{\zeta}\right\}^{1/4}
\left\{ \frac{i}{\Gamma(n + a - 1)} 
w^{(1)}_{n}(uz; a)\mathrm{Ai}_{-1}\left(u^{2/3}\zeta\right) 
\right.
\\
\left. -e^{\pi i/6}w^{(-1)}_{n}(uz; a)\mathrm{Ai}_{1}\left(u^{2/3}\zeta\right) 
\right\},
\end{multline}
where
\begin{equation} 
\label{eq61}
\lambda_{1}=\frac{e^{-\pi i/6}}{\Gamma(n + a - 1)}, \quad
\lambda_{-1}=e^{-\pi i/6}.
\end{equation}

To obtain our desired asymptotic expansions we simply use the LG expansions for the functions involved. For the Airy functions, we have from \cite[Thm. 2.4]{Dunster:2021:SEB} that these LG expansions are given by
\begin{equation}
\label{eq62}
\mathrm{Ai}\left(u^{2/3}\zeta\right) 
\sim \frac{1}{2\pi ^{1/2}u^{1/6}\zeta^{1/4}}
\exp \left\{ -u\xi +\sum\limits_{s=1}^{\infty}
{(-1) ^{s}\frac{a_{s}}{s u^{s}\xi ^{s}}}\right\},
\end{equation}
and 
\begin{equation}
\label{eq63}
\mathrm{Ai}^{\prime }\left( u^{2/3}\zeta\right)
\sim-\frac{u^{1/6}\zeta ^{1/4}}
{2\pi ^{1/2}}\exp \left\{-u\xi
+\sum\limits_{s=1}^{\infty}{(-1)^{s}
\frac{\tilde{a}_{s}}{s u^{s}\xi ^{s}}}\right\},
\end{equation}
as $u \xi \to \infty$, uniformly for $|\arg(\zeta)|\leq \frac{2}{3}\pi$ (or equivalently $|\arg(\xi) | \leq \pi $). These regions are not maximal, but suffice for our purposes (and similarly for the expansions for $\mathrm{Ai}_{\pm 1}(u^{2/3}\zeta)$ and their derivatives). Here 
\begin{equation}
\label{eq64}
a_{1}=a_{2}=\tfrac{5}{72}, \;
\tilde{a}_{1}=\tilde{a}_{2}=-\tfrac{7}{72},
\end{equation}
with $a_{s}$ and $\tilde{{a}}_{s}$ ($s=3,4,5,\ldots $) satisfying the recursion formula
\begin{equation}
\label{eq65}
a_{s+1}=\frac{1}{2}\left(s+1\right) a_{s}+\frac{1}{2}
\sum\limits_{j=1}^{s-1}{a_{j}a_{s-j}} 
\quad (s=2,3,4,\ldots).
\end{equation}

Next, for $z \in Z_{0,1}$ (see \cref{fig:Z0UZ1}), we use the expressions \cref{eq57,eq59} (with upper signs taken for each $\pm1$), and in these we insert the LG expansions \cref{eq36,eq37,eq45,eq46,eq47,eq48} for $w^{(0)}_{n}(uz; a)$ and $w^{(1)}_{n}(uz; a)$, as well as the ones above for the Airy functions. As a result, we arrive at
\begin{multline}
\label{eq66}
\mathcal{A}(u,a,z)
\sim K_{1}(u,\alpha)\exp\left\{\sum\limits_{s=1}^{\infty}
\frac{\tilde{\mathcal{E}}_{s}(a,z)}{u^{s}}\right\}
\\
+K_{2}(u,\alpha)\exp\left\{\sum\limits_{s=1}^{\infty}(-1)^{s}
\frac{\tilde{\mathcal{E}}_{s}(a,z)}{u^{s}}\right\},
\end{multline}
and
\begin{multline}
\label{eq67}
\mathcal{B}(u,a,z)
\sim \frac{1}
{u^{1/3} \zeta^{1/2}}
\left[K_{1}(u,\alpha)\exp\left\{\sum\limits_{s=1}^{\infty}
\frac{\mathcal{E}_{s}(a,z)}{u^{s}}\right\}
\right.
\\
\left.
-K_{2}(u,\alpha)\exp\left\{\sum
\limits_{s=1}^{\infty}(-1)^{s}
\frac{\mathcal{E}_{s}(a,z)}{u^{s}}\right\}
\right],
\end{multline}
where
\begin{equation} 
\label{eq68}
K_{1}(u,\alpha)
=-e^{\pi i/12}e^{(n+a)\pi i/2}\sqrt{\frac{\pi}{2}}
\frac{u^{1/6}}{\Gamma(n+a-1)}\left(\frac{1+\alpha}{4}
\right)^{u/2}\left\{\frac{u(1+\alpha)}{2e}\right\}^{u\alpha/2},
\end{equation}
\begin{equation} 
\label{eq69}
K_{2}(u,\alpha)
=-e^{\pi i/12}e^{(n+a)\pi i/2}\sqrt{\frac{\pi}{2}}
\frac{u^{1/6}}{n!\{4(1+\alpha)\}^{u/2}}
\left\{\frac{e}{2u(1+\alpha)}\right\}^{u\alpha/2},
\end{equation}
\begin{equation} 
\label{eq70}
\tilde{\mathcal{E}}_{s}(a,z)
=\mathrm{E}_{s}(\alpha,\phi)
+(-1)^{s}\frac{\tilde{a}_{s}}{s\xi^{s}},
\end{equation}
and
\begin{equation} 
\label{eq71}
\mathcal{E}_{s}(a,z)
=\mathrm{E}_{s}(\alpha,\phi)
+(-1)^{s}\frac{a_{s}}{s\xi^{s}}.
\end{equation}

If we follow the same procedure using \cref{eq57,eq59} (with lower signs taken for each $\pm1$) for $z \in Z_{-1,0}$ (see \cref{fig:Zm1UZO}), and using \cref{eq58,eq60} (with lower signs taken for each $\pm1$) for $z \in Z_{-1,1}$ (see \cref{fig:Zm1UZ1}), we obtain the same expansions \cref{eq66,eq67}. Thus, these are valid in $Z_{0,-1}\cup Z_{0,1}\cup Z_{-1,1}$, which upon examination of \cref{fig:Z0UZ1,fig:Zm1UZO,fig:Zm1UZ1} consists of the whole upper half-plane $0 \leq \arg(z) \leq \frac12 \pi$ except for the turning point $z=z_{1}$ (and we shall address this excluded point shortly).

To simplify the expansions, we shall use the identity
\begin{equation} 
\label{eq72}
c_{1}w_{1}+c_{2}w_{2}=\sqrt{c_{1}c_{2}}
\left[\exp\left\{\tfrac12
\ln\left(c_{1}/c_{2}\right)\right\}w_{1}
+\exp\left\{-\tfrac12
\ln\left(c_{1}/c_{2}\right)\right\}w_{2}\right],
\end{equation}
for any $c_{1,2} \neq 0$. In order to do so, from \cref{eq68}, \cref{eq69} and \cite[Eqs. 5.5.5 and 5.11.1]{NIST:DLMF} we have
\begin{multline} 
\label{eq73}
\frac12 \ln \left\{\frac{K_{1}(u,\alpha)}
{K_{2}(u,\alpha)}\right\}
=\frac12 \biggl[u\alpha(\ln(u)-1)+u(1+\alpha)\ln(1+\alpha) 
\biggr.
\\ 
\left.
+\ln\left\{\Gamma\left(u+\frac12\right)\right\}
-\ln\left\{\Gamma\left(u+u\alpha+\frac12\right)\right\}\right]
\sim \sum_{s=0}^{\infty}\frac{d_{2s+1}(\alpha)}{u^{2s+1}}
\quad (u \to \infty),
\end{multline}
observing that this series only involves odd powers of $1/u$. The first four terms are found to be
\begin{equation} 
\label{eq74}
d_{1}(\alpha)=-\frac{\alpha}{48(1+\alpha)},
\end{equation}
\begin{equation} 
\label{eq75}
d_{3}(\alpha)=\frac{7\alpha\left(3+3\alpha
+\alpha^{2}\right)}{5760(1+\alpha)^{3}},
\end{equation}
\begin{equation} 
\label{eq76}
d_{5}(\alpha)=-\frac{31\alpha\left(5+10\alpha
+10\alpha^{2}+5\alpha^{3}+\alpha^{4}\right)}
{80640(1+\alpha)^{5}},
\end{equation}
and
\begin{equation} 
\label{eq77}
d_{7}(\alpha)
=\frac{127\alpha\left(7+21\alpha+35\alpha^{2}
+35\alpha^{3}+21\alpha^{4}+7\alpha^{5}+\alpha^{6}\right)}
{430080(1+\alpha)^{7}}.
\end{equation}
Then, from \cref{eq66,eq67,eq68,eq69,eq72,eq73}, as $u \to \infty$,
\begin{equation} 
\label{eq78}
\mathcal{A}(u,a,z)=\frac{e^{5\pi i/6}e^{(u+a)\pi i/2}
u^{1/6}}{2^{n-1}}
\left\{\frac{\pi }{2^{a}n!\,\Gamma(n+a-1)}
\right\}^{1/2}A(u,a,z),
\end{equation}
and
\begin{equation} 
\label{eq79}
\mathcal{B}(u,a,z)=\frac{e^{5\pi i/6}e^{(u+a)\pi i/2} 
u^{1/6}}{2^{n-1}}
\left\{\frac{\pi }{2^{a}n!\,\Gamma(n+a-1)}
\right\}^{1/2}B(u,a,z),
\end{equation}
where
\begin{equation} 
\label{eq80}
A(u,a,z) \sim 
\exp \left\{ \sum\limits_{s=1}^{\infty}
\frac{\tilde{\mathcal{E}}_{2s}(a,z) }{u^{2s}}\right\} 
\cosh \left\{ \sum\limits_{s=0}^{\infty}
\frac{\tilde{\mathcal{E}}_{2s+1}(a,z)
+d_{2s+1}(\alpha)}
{u^{2s+1}}\right\},
\end{equation}
\begin{equation} 
\label{eq81}
B(u,a,z) \sim \frac{1}{u^{1/3}\zeta^{1/2}}
\exp \left\{ \sum\limits_{s=1}^{\infty}
\frac{\mathcal{E}_{2s}(a,z) }{u^{2s}}\right\} 
\sinh \left\{ \sum\limits_{s=0}^{\infty}
\frac{\mathcal{E}_{2s+1}(a,z)
+d_{2s+1}(\alpha)}
{u^{2s+1}}\right\},
\end{equation}
uniformly for $0 \leq \arg(z) \leq \pi$ with $|z-z_{1}| \geq \delta >0$, under the condition \cref{eq17}.

\begin{lemma}
\label{lem:odd-mero}
Each $(z-z_{1})^{1/2}\{\mathrm{E}_{2s+1}(\alpha,\phi)+d_{2s+1}(\alpha)\}$ ($s=0,1,2,\ldots$), regarded as a function of $z$, is meromorphic at $z=z_{1}$.
\end{lemma}

\begin{proof}
From \cite{Dunster:2017:COA} the even coefficients $\mathrm{E}_{2s}(\alpha,\phi)$ ($s=1,2,3,\ldots$) are meromorphic at $z=z_{1}$. In general, the same is only true for each $(z-z_{1})^{1/2}\mathrm{E}_{2s+1}(\alpha,\phi)$ ($s=0,1,2,\ldots$) if the integration constants in the recursion relation (in our case \cref{eq34} for $s$ even) are suitably chosen, i.e. with appropriate constants added to $\mathrm{E}_{2s+1}(\alpha,\phi)$.  Now, with the aid of Maclaurin series for $\sinh(z)$ (\cite[Eq. 4.33.1]{NIST:DLMF}), we formally expand \cref{eq81} as an asymptotic series in inverse powers of $u^2$. This series holds in a punctured neighbourhood of $z=z_{1}$ ($\zeta=\xi=0$). From this we see that each $(z-z_{1})^{1/2}\{\mathcal{E}_{2s+1}(\alpha,\phi)+d_{2s+1}(\alpha)\}$ must be single-valued in this disk, since $B(u,a,z)$ is analytic at that point, as well as $(z-z_{1})^{1/2}\zeta^{-1/2}$. Now each $(z-z_{1})^{1/2}\xi^{-2s-1}$ is meromorphic at $z=z_{1}$ (see \cref{eq20}), and hence the result follows from \cref{eq71}.
\end{proof}

As an example, we have from \cref{eq16,eq27,eq22,eq29,eq32,eq74}
\begin{multline} 
\label{eq82}
\sqrt{2\sigma i}\left(z-z_{1}\right)^{1/2}
\left\{\mathrm{E}_{1}(\alpha,\phi)+d_{1}(\alpha)\right\}
=\frac{5(\sigma+i)^{2}i}{96\left(z-z_{1}\right)}
\\
+\frac{\alpha-6\sigma i}{128\sigma}
+\frac{18\sigma+31\alpha i}{1024(1+\alpha)}\left(z-z_{1}\right)
+\dots
\quad (z \to z_{1}),
\end{multline}
where $\sigma$ is given by \cref{eq28}.

In summary, from \cref{eq04,eq52,eq54,eq78,eq79}, we have for the reverse generalized Bessel polynomials our main asymptotic result. 
\begin{theorem}
\label{thm:ThetaAiry}
Let $\alpha$, $f(a,z)$, $z_{1}(\alpha)$, $\xi$ and $\zeta$ be given by \cref{eq14,eq15,eq16,eq20,eq21,eq22}. Then
\begin{multline} 
\label{eq83}
\theta_{n}(uz;a)
=\frac{u^{1/6}}{e^{(u+a+2)\pi i/2}}
\left\{\frac{2^{a} \pi n! }{\Gamma(n+a-1)}
\right\}^{1/2}
\\ \times
\left\{\frac{\zeta}{f(a,z)}\right\}^{1/4}
(uz)^{n+\frac{1}{2}a-1}  e^{uz} 
\left\{
\mathrm{Ai}\left(u^{2/3}\zeta\right)A(u,a,z)
+\mathrm{Ai}'\left(u^{2/3}\zeta \right)B(u,a,z)
\right\},
\end{multline}
where, as $u=n+\frac12 \to \infty$, the coefficient functions $A(u,a,z)$ and $B(u,a,z)$ possess the asymptotic expansions \cref{eq80,eq81}, uniformly for $\Im(z) \geq 0$ such that $|z-z_{1}|\geq \delta >0$, and $-1 <-1+\delta \leq \alpha \leq \alpha_{1} < \infty$. In these, $\tilde{\mathcal{E}}_{s}(a,z)$ and $\mathcal{E}_{s}(a,z)$ are given by \cref{eq22,eq27,eq29,eq32,eq33,eq34,eq35,eq64,eq65,eq70,eq71}, in which the constants $d_{2s+1}(\alpha)$ ($s=0,1,2,\ldots$) are those appearing in the expansion \cref{eq73}, with the first four given by \cref{eq74,eq75,eq76,eq77}.
\end{theorem}

\begin{remark}
$A(u,a,z)$ and $B(u,a,z)$ are analytic for $\Im(z) \geq 0$, including at the turning point $z_{1}$, but the expansions \cref{eq80,eq81} are not valid at this point. Asymptotic approximations that are valid for $|z-z_{1}|\leq \delta$ can be achieved either using the above expansions in conjunction with Cauchy's integral formula, as described in \cite[Thm. 4.2]{Dunster:2021:SEB}, or by expanding them as regular expansions involving inverse powers of $u$, as follows. Similarly to \cite[Eqs. (4.9) and (4.20)]{Dunster:2024:AEG} let $\tilde{q}_{1}(a,z)=\tilde{\mathcal{E}}_{1}(a,z)$ and
\begin{equation}
\label{eq84}
\tilde{q}_{s}(a,z)=\tilde{\mathcal{E}}_{s}(a,z)
+\frac{1}{s}\sum_{j=1}^{s-1}j
\tilde{\mathcal{E}}_{j}(a,z)\tilde{q}_{s-j}(a,z)
\quad (s=2,3,4,\ldots),
\end{equation}
and $q_{1}(a,z)=\mathcal{E}_{1}(a,z)$ with
\begin{equation}
\label{eq85}
q_{s}(a,z)=\mathcal{E}_{s}(a,z)
+\frac{1}{s}\sum_{j=1}^{s-1}j
\mathcal{E}_{j}(a,z)q_{s-j}(a,z)
\quad (s=2,3,4,\ldots).
\end{equation}
Then for $\mathrm{A}_{s}(a,z)=\tilde{q}_{2s}(a,z)$ ($s=1,2,3,\ldots$) and $\mathrm{B}_{s}(a,z)=\zeta^{-1/2}q_{2s+1}(a,z)$ ($s=0,1,2,\ldots$)
\begin{equation}
\label{eq86}
A(u,a,z) \sim 
1+\sum\limits_{s=1}^{\infty}
\frac{\mathrm{A}_{s}(a,z)}{u^{2s}}
\quad (u \to \infty),
\end{equation}
and
\begin{equation}
\label{eq87}
B(u,a,z) \sim \frac{1}{u^{4/3}}
\sum\limits_{s=0}^{\infty}
\frac{\mathrm{B}_{s}(a,z)}{u^{2s}}
\quad (u \to \infty).
\end{equation}
as $u \to \infty$, uniformly for $\Im(z) \geq 0$ and $-1 <-1+\delta \leq \alpha \leq \alpha_{1} < \infty$. Here, the coefficients $\mathrm{A}_{s}(a,z)$ and $\mathrm{B}_{s}(a,z)$ have a removable singularity at $z=z_{1}$ as a consequence of \cite[Thm. 3.1]{Dunster:2021:NKF}.
\end{remark}

\begin{figure}
 \centering
 \includegraphics[
 width=1.0\textwidth,keepaspectratio]
 {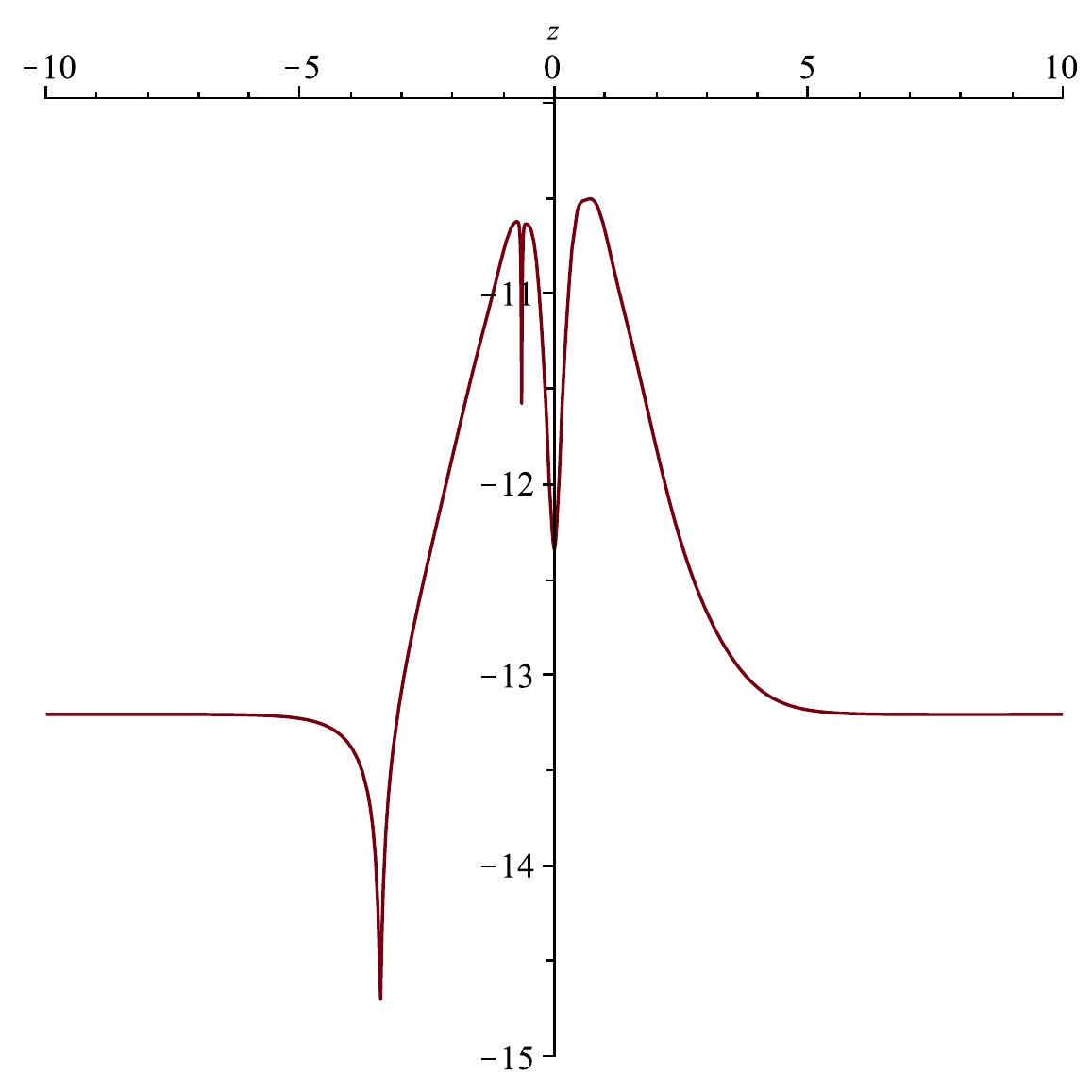}
 \caption{Graph of $\Omega(u,a,z)$ for $a=1.2$, $u=20.5$ and $-10 \leq z \leq 10$.}
 \label{fig:Omega}
\end{figure}

To demonstrate the relative accuracy of \cref{eq83} we plot in \cref{fig:Omega} the graph of
\begin{equation}
\label{eq88}
\Omega(u,a,z)=\log_{10}\left| \frac{\theta_{n}(uz;a)
-\Theta_{3}(u,a,z)}{\theta_{n}(uz;a)} \right|,
\end{equation}
for $a=1.2$, $u=20.5$ ($n=20$) and $-10 \leq z \leq 10$; here $\Theta_{3}(u,a,z)$ is the approximation given by the RHS of \cref{eq83} with $A(u,a,z)$ and $B(u,a,z)$ replaced by $A_{3}(u,a,z)$ and $B_{3}(u,a,z)$, respectively, where
\begin{equation} 
\label{eq89}
A_{3}(u,a,z) =
\exp \left\{ \sum\limits_{s=1}^{3}
\frac{\tilde{\mathcal{E}}_{2s}(a,z) }{u^{2s}}\right\} 
\cosh \left\{ \sum\limits_{s=0}^{2}
\frac{\tilde{\mathcal{E}}_{2s+1}(a,z)
+d_{2s+1}(\alpha)}
{u^{2s+1}}\right\},
\end{equation}
and
\begin{equation} 
\label{eq90}
B_{3}(u,a,z)= \frac{1}{u^{1/3}\zeta^{1/2}}
\exp \left\{ \sum\limits_{s=1}^{3}
\frac{\mathcal{E}_{2s}(a,z) }{u^{2s}}\right\} 
\sinh \left\{ \sum\limits_{s=0}^{2}
\frac{\mathcal{E}_{2s+1}(a,z)
+d_{2s+1}(\alpha)}
{u^{2s+1}}\right\}.
\end{equation}

\begin{figure}
 \centering
 \includegraphics[
 width=1.0\textwidth,keepaspectratio]
 {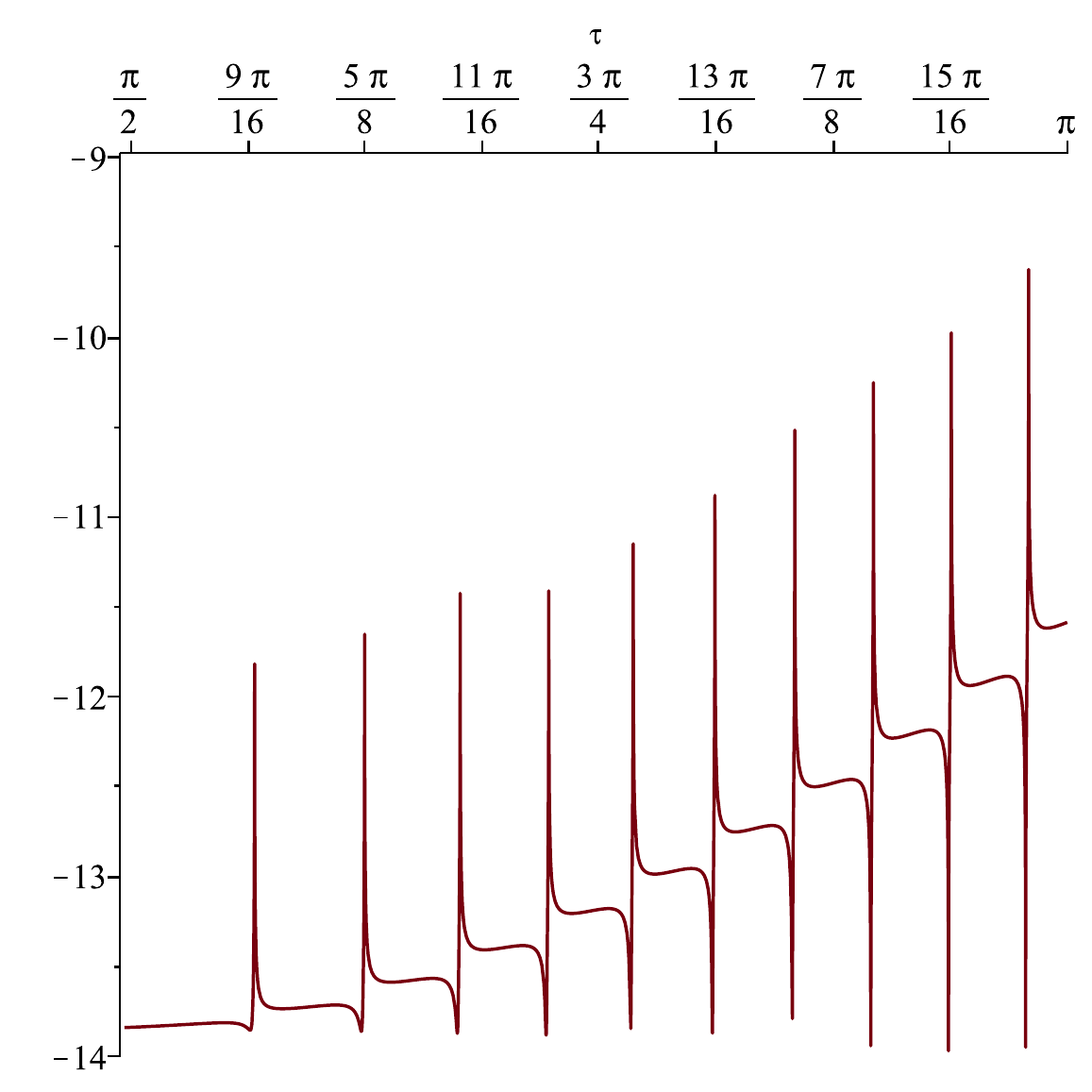}
 \caption{Graph of $\hat{\Omega}(u,a,\tau)$ for $a=1.2$, $u=20.5$ and $\tau_{0} \leq \tau \leq \pi$.}
 \label{fig:OmegaSL}
\end{figure}

\begin{figure}
 \centering
 \includegraphics[
 width=1.0\textwidth,keepaspectratio]
 {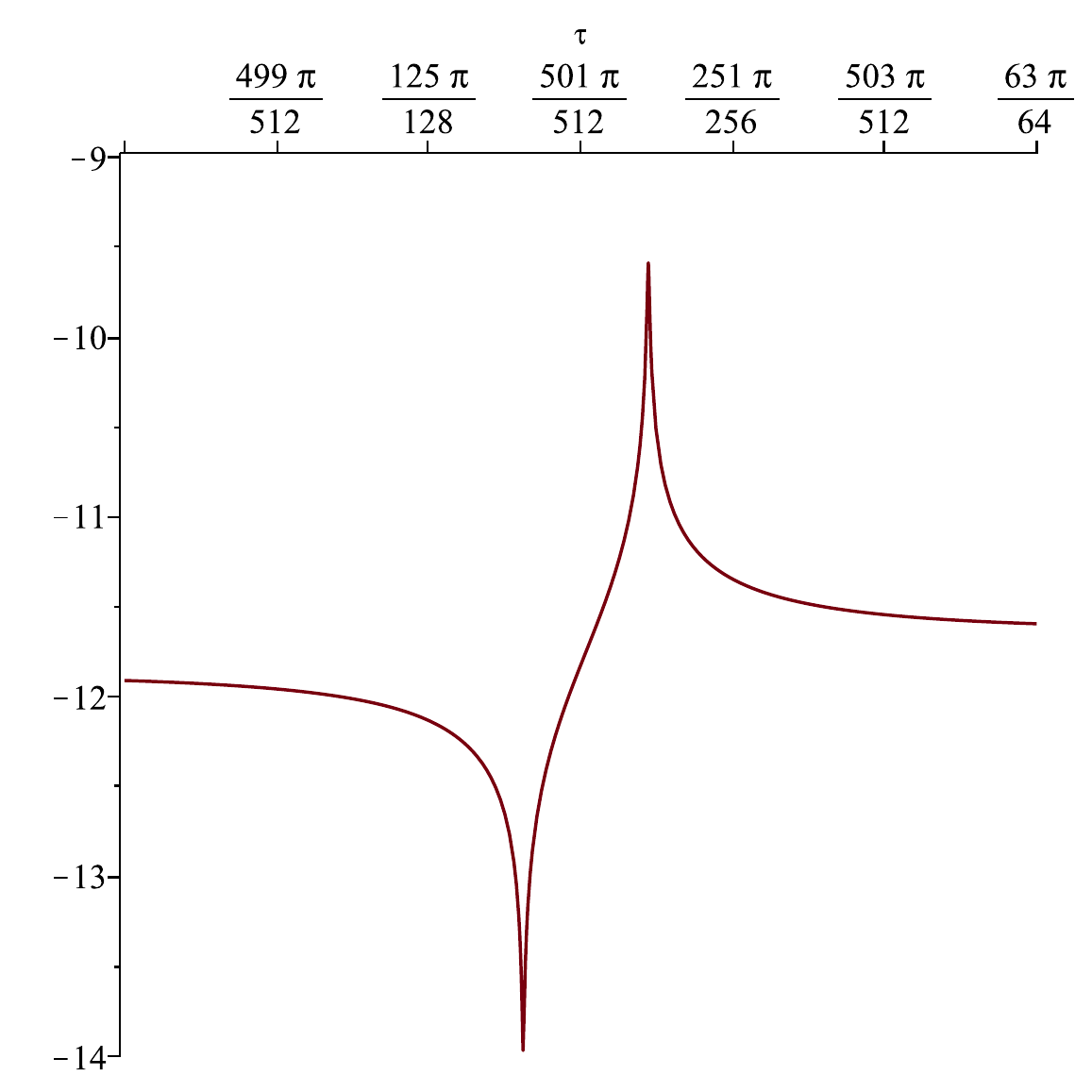}
 \caption{Graph of $\hat{\Omega}(u,a,\tau)$ for $a=1.2$, $u=20.5$ and $\frac{249}{256}\pi \leq \tau \leq \frac{63}{64}\pi$.}
 \label{fig:OmegaSL1}
\end{figure}

Next, consider $z$ lying on the Stokes curve $\Re(\xi)=0$, which in  \cref{fig:Figzplane} is labeled $\mathsf{A}\mathsf{H}$. The zeros of $\theta_{n}(uz;a)$ in the upper half-plane lie close to this curve, but not exactly on it. On this curve $z$ runs from $z_{1}=-\frac12 \alpha+i \sigma$ to $x_{1}$, where the latter value is the negative root of \cref{eq21} when $\xi=-\frac12 \pi i$. Let $\tau=\tau_{0} = \arg(z_{1})$, so that 
\begin{equation}
\label{eq91}
\tau_{0}=\frac12 \pi+\arctan\left(\frac{\alpha}{2\sigma}\right).
\end{equation}
We then parameterise the curve by $z=z(\tau)=r(\tau)e^{i\tau}$ where $\tau_{0} \leq \tau \leq \pi$, and $x_{1} \leq r \leq |z_{1}|=1+\frac12 \alpha$.

Since the turning point $z_{1}$ lies on this curve we use \cref{eq86,eq87} instead of \cref{eq89,eq90}, and in particular the following truncated versions of these expansions
\begin{equation}
\label{eq92}
\hat{A}_{3}(u,a,z) =
1+\sum\limits_{s=1}^{3}
\frac{\mathrm{A}_{s}(a,z)}{u^{2s}}, \,
\hat{B}_{3}(u,a,z) = \frac{1}{u^{4/3}}
\sum\limits_{s=0}^{3}
\frac{\mathrm{B}_{s}(a,z)}{u^{2s}}.
\end{equation}

With these, we define $\hat{\Theta}_{3}(u,a,z)$ as the approximation given by the RHS of \cref{eq83} with $A(u,a,z)$ and $B(u,a,z)$ replaced by $\hat{A}_{3}(u,a,z)$ and $\hat{B}_{3}(u,a,z)$, respectively. To investigate the relative error, define 
\begin{equation}
\label{eq93}
\hat{\Omega}(u,a,\tau)=\log_{10}\left| \frac{\theta_{n}(uz(\tau);a)
-\hat{\Theta}(u,a,z(\tau))}{\theta_{n}(uz(\tau);a)} \right|,
\end{equation}
where $z(\tau)=r(\tau)e^{i\tau}$. In \cref{fig:OmegaSL} we plot this function for $a=1.2$, $u=20.5$ and $\tau_{0} \leq \tau \leq \pi$, and one observes high accuracy along the curve. Points having cusps are explained as being close to the zeros of $\theta_{n}(uz;a)$ (which is in the denominator of \cref{eq93}). In \cref{fig:OmegaSL1} we show in more detail the graph at points near the zero which is closest to the negative real axis ($\tau \approx \pi$).

\newpage

\section*{Acknowledgment}
Financial support from Ministerio de Ciencia e Innovación pro\-ject PID2021-127252NB-I00 (MCIN/AEI/10.13039/ 501100011033/FEDER, UE) is acknowledged.

\section*{Conflict of interest}
The author has no conflict of interest to declare that is relevant to the content of this article.

\bibliographystyle{siamplain}
\bibliography{biblio}
\end{document}